\documentclass[a4paper,11pt]{article}
\usepackage[utf8x]{inputenc}
\PrerenderUnicode{é}
\usepackage[title]{appendix}
\usepackage{varioref}
\usepackage{etex}
\usepackage{geometry}
\usepackage{mathtools}
\usepackage{amsfonts,amsmath,amssymb}
\usepackage{xmpmulti}
\usepackage{multicol}
\usepackage{hyperref}
\title{$3n+1$ problem: a heuristic lower bound for the number of integers connected to 1 and less than $x$}

\author{
  Daudin, Jean-Jacques\footnote{jeanjacques.daudin@gmail.com} \\
  \texttt{Université Paris-Saclay,AgroParisTech,INRAE,UMR MIA-Paris,75005,Paris,France }
}

\usepackage{amsthm}
\newtheorem{theorem}{Theorem}
\newtheorem{definition}{Definition}

\newtheorem{proposition}{Proposition}

\newtheorem*{summary}{Summary}

\begin{document}

\maketitle

\begin{summary}
This paper gives a heuristic lower bound for the number of integers connected to 1 and less than $x$, $\theta(x) > 0.9x,$ in the context of the $3x+1$ problem.   
\end{summary}

\section{Basic elements}
In the presentation of the book "The Ultimate Challenge: The 3x+1 Problem", \cite{Lagarias}, J.C. Lagarias write {\it The $3x+1$ problem, or Collatz problem, concerns the following seemingly innocent arithmetic procedure applied to integers: If an integer $x$ is odd then "multiply by three and add one", while if it is even then "divide by two". The $3x+1$ problem asks whether, starting from any positive integer, repeating this procedure over and over will eventually reach the number 1. Despite its simple appearance, this problem is unsolved.} We refer to this book and other papers from the same author for a review of the context and the references.

\subsection{Definitions}
Let $n \in \mathbb{N}$.
\subsubsection*{Direct algorithm}

$$ T(n)=\left\lbrace \begin{array}{ll}
3n+1 & if \quad n\equiv 1 \pmod2 \\
n/2 & if \quad n\equiv 0 \pmod2 \\
\end{array}
 \right.
$$
\subsubsection*{Inverse algorithm}
$$ U(n)=\left\lbrace 
2n \quad  and \quad \dfrac{n-1}{3} \quad   if \quad n\equiv 4 \pmod6 
\right\rbrace
$$

\subsubsection*{Conjecture "$3x+1$"}
$\forall n \in \mathbb{N}, \exists k \in \mathbb{N} :T^k(n)=1.$

\subsection{Restriction to odd integers}

\subsubsection*{ $f$ and  $h$} \label{def f et h}

If the "$3x+1$" conjecture is true for the odd integers it is also true for the even ones by definition of $T$. The expressions of $T$ and $U$ restricted to odd terms are the following with $n$ odd:
\begin{itemize}
\item $T$ becomes $f$: $f(n)=(3n+1)2^{-j(3n+1)}$ with $j(3n+1)$ the power of 2 in the prime factors decomposition of $3n+1$. $f$ is often called the "Syracuse function".
\item $U$ becomes $h$, see\cite{Colussi}:
$$ h(n)=\left\lbrace \begin{array}{ll}
\emptyset & {\rm if} \quad n\equiv 0 \pmod3 \\
{ \frac{n2^k-1}3, k=2,4,6...} & {\rm if} \quad n\equiv 1 \pmod3 \\
{ \frac{n2^k-1}3, k=1,3,5...} & {\rm if} \quad n\equiv 2 \pmod3 \\
\end{array}
 \right.
$$
\end{itemize}

\subsubsection*{Graph $g(n)$}

Let $(n_1,n_2)$ be odd integers. $n_1$ and $n_2$ are connected by an edge if $n_1=f(n_2)$ or $n_2=f(n_1)$. $g(n)$ is the subset of the odd integers connected to n. $g_b(n)$ is the subset of the odd integers connected to n by a chain containing exactly $b+1$ odd numbers (including $1$ and $n$).

\section{Properties of $g(1)$}

\subsection{Expression of $n \in g(1)$ as a sum of fractions}

\begin{proposition} \label{theoremFractions}
Let $n \in g(1). \; \exists  (b,a>u_1>u_2,...>u_b=0) \in \mathbb{N}^{b+2} :\;$
$$ n=\frac{2^a}{3^b}-\sum_{i=1,b}\frac{2^{u_i}}{3^{b-i+1}}.$$
\end{proposition}

 Note that $\frac{2^a}{3^b} \geq 1 \Rightarrow a \geq b\frac{log3}{log2}.$
 \begin{proof}
 See \cite{Daudin}
 \end{proof}

\subsection{Admissible tuple $(b,a>u_1>u_2,...>u_b=0)$}
Only some values of $(b,a>u_1>u_2,...>u_b=0)$ give an integer $n$ in theorem 1, most of them do not. 

\begin{definition} \label{Deftuple}
A tuple $(b,a \geq b\frac{log3}{log2}, a>u_1>u_2,...>u_b=0)$  is admissible if $\frac{2^a}{3^b}-\sum_{i=1,b}\frac{2^{u_i}}{3^{b-i+1}} \in \mathbb{N}.$
\end{definition}
In the following we use the alternative notation $(b,v_1,v_2,...,v_b)$ for the tuple $ (b,a=u_0>u_1>u_2,...>u_b=0) ,$ with $v_i=u_{i-1}-u_{i}, \; i=1,...b.$ 
\begin{equation*}
(b,\sum_{i=1,b}v_i, \sum_{i=2,b}v_i, \sum_{i=3,b}v_i,..., v_b,0)=(b,u_0>u_1>u_2,...>u_b=0) .
\end{equation*}
 In few words, $b+1$ is the number of odd integers in the chain from $1$ to $n$, $v_i$ is the number of divisions by $2$ at the $(b-i)^{th}$ step of $f$ (the exponent of $2$ at the $i^{th}$ step of $h$) and $a=\sum v_i.$
The tuple $(v_1,...v_b)$ is admissible if and only if
 
\begin{equation} \label{admissible}
 2^{\sum_{1}^{b} v_i}\equiv \sum_{i=1,b-1}2^{\sum_{i+1}^{b} v_j}3^{i-1} +3^{b-1}\pmod{3^b}.
 \end{equation} 
 
Let $g_b^*(1)=\{ admissible-tuples(v_1,v_2,...v_b) \; : v_i \le m=2.3^{b-1}, i=1,b\}.$ 

 The Wirsching-Goodwin representation of $g_b^*(1)$ (see \cite{Goodwin}, \cite{Daudin}) gives the whole structure of the $v_i$s. Its expression is the following: 

 \begin{theorem} \label{theoremGoodwin2}
There is a one to one relation between $g_b^*(1)$ with $b>1$ and the set of the  t-uples $ (b,v'_1,v'_2,...,v'_b)$ with $v'_i=v_i+2.3^{b-i}c_i,$ $c_i \in \mathbb{N^*}$, $v_i \in \mathbb{N}, \; i=2,...b  \; with \; 1 \leq v_i \leq 2.3^{b-i} $ and $4 \leq v_1 \leq 2.3^{b-1}+2$ is the unique
 solution of equation (\ref{admissible}).
\end{theorem}

Therefore for each $(b$ and $(v_2,...v_b) \in (1,2.3^{b-1})^{b-1},$ there is a unique $v_1 \in  (4, 2.3^{b-1}+2)$  such that $(v_1,v_2,...v_b)$ is admissible.

\section{Outline}
Krasikov \cite{Krasikov} proved that $\theta(x) >cx^{3/7}$, with  $\theta(x)=\# \{u:T^k(u)=1, k \ge 0, u<x \},$ and $c$ is a constant. This result has been improved by Applegate and Lagarias \cite{Applegate} : $\theta(x) >x^{0.81}$ and then by Krasikov and Lagarias \cite{Krasikov2} :

\begin{equation}
\label{KL}
\theta(x) >x^{0.84}.
\end{equation}

 This is the best bound obtained till now for $\theta(x)$. A significative lower bound to say something new for the "$3x+1$ problem" would be $\theta(x) >Cx.$  

\bigskip 

The heuristic proposed in this paper is 
$$ \theta(x)\geqslant   \frac{3}{8}\frac{1}{2-\log_2(3)}  x \simeq 0.9035. x$$

The path to set this proposition has three steps. The steps 1 and 3 are well established results. The step 2 contains a lower bound that is not proved but seems to be true and can perhaps be proved with some more work.
\begin{enumerate}
\item{\bf{Step 1.}} 
The inequality $n \leq x$ is replaced by the little more stronger one $a(n) \leq a(x)$ which is more tractable.

 Let $n \in g(1)$, $(b,v_1,v_2,...,v_b)$ the corresponding t-uple, and  $a=\sum_{i=1,b}v_i$. 
 
 Let $a(x)=\log_2(x)+b\log_2(3)$. The key point is
  \begin{equation} \label{key}
 a \leq a(x) \Rightarrow \frac{2^a}{3^b} \leq x \Rightarrow n=\frac{2^a}{3^b}-\sum_{i=1,b}\frac{2^{u_i}}{3^{b-i+1}} < x
 \end{equation}
Let $N(b,x) = \#\{ n \in g(1) \; : \; b(n)=b, \; \; n\leq x\}$ be the number of odd integers less than $x$ and reached in $b$ steps.

Let $M(b,x) = \#\{ n \in g(1) \; : \; b(n)=b, \; \; a(n) \leq a(x) \}$ be the number of odd integers reached in $b$ steps and such that $a(n) \leq a(x)$. (\ref{key}) implies that 

\begin{equation*}
M(b,x) \leq N(b,x)
\end{equation*} 

and 
\begin{equation}
\sum_{b=1}^{\infty}M(b,x) \leq \sum_{b=1}^{\infty}N(b,x)=\frac{\theta(x)}{2}
\end{equation}
\item{\bf{Step 2.}}
For fixed $b$, $a$ behaves approximatively as the sum of $b$ independent uniform variables, 
$$M(b,x) \simeq \frac{3}{2}\binom{a(x)-4}{b}3^{-b},$$ and the proposed but yet unproved inequality: $$ \frac{\theta(x)}{2} \geq 1+\sum_{b=1}^{\infty}\frac{3}{2}\binom{a(x)-4}{b}3^{-b}.$$
\item{\bf{Step 3.}}
  $$\sum_{b=0}^{\infty}\frac{3}{2}\binom{a(x)-4}{b}3^{-b}=\frac{3}{16}\frac{1}{2-\log_2(3)}  x$$
\end{enumerate}

\section{Step 2: ${\theta(x)/2} \geq \sum_{b=1}^{\infty}\frac{3}{2}\binom{a(x)-4}{b}3^{-b}.$ }  
First let us recall some results about the pdf of the sum of uniform variables on integers.
\subsection{Pdf of the sum of uniform variables on integers}

Let $U_m$ be the uniform pdf on integers $(1,m)$, with $P(X=i)=\frac{1}{m}.$ Let $X_1,...X_b$ be $b$ independent variables with $X_i \sim U_m.$ The probability generating function $q_b(s)$ of $S_b=\sum_{i=1}^bX_i$ is $$q_b(s)=\left[\frac{1}{m}\sum_{i=1,m}s^i \right]^b $$
with $E(S_b)=b\frac{m+1}{2}$ and $V(S_b)=\frac{b(m^2-1)}{12},$ and 

\begin{eqnarray*} 
P_{SU(b)}(S=a) & = & \frac{q_b^{(a)}(0)}{a!} \\
 & = & \frac{1}{m^b}\sum_{{\tiny \begin{array}{c}
 n_1 \geq 0,...n_m \geq 0\\
 n_1+...n_m =b, \\
  n_1+2n_2+...mn_m=a
  \end{array}}}
  \left( \begin{array}{c}
 b \\
  n_1 \; n_2 \;...n_m
\end{array}   \right) \\
 & = & \frac{1}{m^b} {\binom{b}{a}}_m 
\end{eqnarray*}
$ {\binom{b}{a}}_m $ is the polynomial or extended binomial coefficient\footnote{This definition is different from the usual one. The usual definition of $\binom{b}{a}_m$ is with $U_m$ the uniform pdf on integers $(0,m)$} , see \cite{Neuschel}, that has no closed expression but can be computed  by convolution, using the relation
\begin{equation} \label{recurence}
 {\binom{b}{a}}_m=\sum_{i=1,m}{\binom{b-1}{a-i}}_m
 \end{equation} 

 An integer composition of a nonnegative integer $n$ with $k$ summands, or parts, is a way of writing $n$ as a sum of $k$ nonnegative integers, where the order of parts is significant. 
 A classical result in combinatorics is that the number of S-restricted integer compositions
  of $n$ with $k$ parts is given by the coefficient of $x^n$ of the polynomial or power series $(\sum_{i \in S} x^i)^k, $ which is the extended binomial coefficient, see (\cite{Eger}). The restriction considered in this paper is  $S=(1,m)$. Therefore ${\binom{b}{a}}_m$ is the number of compositions of $a$ in $b$ parts restricted to lay in $(1,m).$

Although $\binom{b}{a}_m$ does not possess a closed form expression, it possesses one in the "no-constraint" particular case defined by condition C1:

\textbf{Condition C1}:  $$a \leq m+b-1$$ 

\begin{proposition}\label{CoefExact}
if  C1 is true 
\begin{eqnarray*}
& (i) & \binom{b}{a}_m =\binom{a-1}{b-1} \\
& (ii) & \sum_{j=b,a} \binom{b}{j}_m =\binom{a}{b}
 \end{eqnarray*}  

\end{proposition}

\begin{proof}
(i) is the integer composition of the positive integer $a$ with $b$ summands, without any constraint on the summands.
The proof of (ii) comes from (\vref{Uppersum}).
\begin{equation} \label{Uppersum}
\sum_{l=0,n} \binom{l}{k}= \binom{n+1}{k+1}
\end{equation}
We use the convention  $l<k \Rightarrow \binom{l}{k}=0.$ 
\end{proof}

\subsection{Relation between the $3x+1$ problem and the sum of uniform variable on integers}
\subsubsection{Lower bound}
 The Wirsching-Goodwin representation of the odd numbers connected to 1 in $b$ (odd numbers)-steps (see \cite{Goodwin}, \cite{Daudin}) gives the structure of the $v_i$s:
 
Let $(v_1,...v_b) \in g_b^*(1)$ and $a=\sum_{i=1,b} v_i$. The number of elements of $g_b^*(1)$ is $m^{b-1}=2^{b-1}3^{(b-1)^2}=\frac{3}{2}\left(\frac{m}{3}\right)^b$. The $(b-1)$-order contingency table $(v_2 \times v_3...\times v_b)$ is composed of ones in each cell, so the set of random variables, if one pick up a cell at random with the same probability $m^{-b+1}$ for each cell, $(v_2,v_3,...v_b)$ is uniformly and independently distributed. There are two differences between $a$ and the sum of $b$ uniform and independent variables on $(1,m)$:
\begin{itemize}
\item Let $V=\left\{v \in  (4,m+2),\: \mod(v,2)=0 \: \mod(v,3) \neq 0 \right\}$. The variable $v_1$ is uniformly distributed on $V$  in place of $(1,m)$. The sum of uniform variables have thus to be suited to this particular $v_1$ by modifying the initialisation of the convolution equation (\ref{recurence}): the vector with $m$ ones in positions $(1,m)$, is replaced by the vector $(0,0,0,3,0,0,0,3,0,3,...)$ with $3$ on the $m/3$ positions of $V$. Let $C(a,b,m)$ be the resulting modified extended binomial coefficient.
\item $v_1$ is not independent of $(v_2,...v_b)$ because $(v_2,...v_b)$ determinates $v_1$. 
We study here the impact of this dependency on the distribution of $a$. This is the more difficult point of the paper, and not yet proved. We use a heuristic inequality.
\end{itemize}

Let
\begin{eqnarray*}
N_b(v,a_1)& = & \#\{n \in g_b(1), \: v_1=v, a-v_1=a_1 \} \\
 & = & \#\{admissible-tuples(v_1,...v_b): \: v_1=v, \sum_{i=2,b}v_i=a_1 \},
\end{eqnarray*}
The margins of $N_b(v,a_1)$ are
 $$N_b(v)=\sum_{a1=b-1}^{m(b-1)}N_b(v,a_1)=\frac{m^{b-1}}{m/3}$$
 and  
 $$N_b(a_1)=\sum_{v \in V}N_b(v,a_1)=\binom{b-1}{a_1}_m.$$
  $N_b(v)$ does not depend on $v \in V$. However  $N_b(v,a_1)$ depends on $v$ and $a_1$. Let $$\overline{N_b(a_1)}=\frac{N_b(a_1)}{m/3}=\frac{\binom{b-1}{a_1}_m}{m/3},$$ the mean value of $N_b(v,a_1)$ with $a_1$ fixed and $v \in V.$ Let
  $$\alpha_b(v,a_1)=\frac{N_b(v,a_1)}{\overline{N_b(a_1)}}. $$
  Now we can express $M(b,x) = \#\{ n \in g_b(1) \; :  \; a(n) \leq a(x) \}$ using $N_b(v,a_1)$:
  \begin{eqnarray*}
  M(b,x) & = & \sum_{v+a_1 \leq a(x)}N_b(v,a_1) \\
         & = & \sum_{v+a_1 \leq a(x)}\alpha_b(v,a_1)\overline{N_b(a_1)} \\
         & = & 3/m\sum_{v+a_1 \leq a(x)}\alpha_b(v,a_1)\binom{b-1}{a_1}_m \\
  \end{eqnarray*}

   If Condition C1 ($a(x) < m+b-1$)  is true,
   \begin{eqnarray*}
   M(b,x) & = & \frac{3}{m}\sum_{v+a_1 \leq a(x)}\alpha_b(v,a_1)\binom{a_1-1}{b-2} \\
    & = & \frac{3}{m}\sum_{v \in V}\sum_{a_1=b-1}^{a(x)-v}\alpha_b(v,a_1)\binom{a_1-1}{b-2}\\
   \end{eqnarray*} 
    \textbf{Condition C2}: $$ \forall (v,a_1), \; \alpha_b(v,a_1)=1$$
     
    If  C2  is true,
    \begin{eqnarray*}
   M(b,x) & = & \frac{3}{m}\sum_{v \in V}\sum_{a_1=b-1}^{a(x)-v}\binom{a_1-1}{b-2}\\
    & = & \frac{3}{m}\sum_{v \in V}\binom{a(x)-v}{b-1}\\
    & > & \frac{1}{m}\sum_{i=5}^{m+2}\binom{a(x)-i}{b-1} \\
    & > & \frac{1}{m}\binom{a(x)-4}{b}\\
    & > & \frac{3}{2}\binom{a(x)-4}{b}3^{-b}
   \end{eqnarray*} 
   
   The third line comes from the following inequalities:
\begin{eqnarray*}
 3\binom{a(x)-4}{b-1} & > & \binom{a(x)-5}{b-1}+ \binom{a(x)-6}{b-1} + \binom{a(x)-7}{b-1} \\
 3\binom{a(x)-8}{b-1} & > & \binom{a(x)-8}{b-1} +\binom{a(x)-9}{b-1} + \binom{a(x)-10}{b-1} \\
 3\binom{a(x)-10}{b-1} & > & \binom{a(x)-11}{b-1}+ \binom{a(x)-12}{b-1} + \binom{a(x)-13}{b-1} \\
 3\binom{a(x)-14}{b-1} & > & \binom{a(x)-14}{b-1}+ \binom{a(x)-15}{b-1} + \binom{a(x)-16}{b-1} \\
 ...
 \end{eqnarray*} 
   Let $b=\log(\log_2(x))+1$. Therefore $2\times3^{b-1}=2\times3^{\log(\log_2(x))}>2\log_2(x)$. 
   
   $x>4 \Rightarrow \log_2(x)>(\log(\log_2(x))+1)\log_2(3/2)+1$. Therefore $2\times3^{b-1}>b\log_2(3)-b+1+\log_2(x)$ and if $b > \log(\log_2(x))+1,$ the condition C1 is achieved. The maximum number of odd numbers less than $x$ is obtained for $b \simeq 2log_2(x),$ and most of them are obtained with $\log_2(x)/4 \leq b \leq 4log_2(x).$  $\sum_{b=1}^{\log(\log_2(x))}\frac{3}{2}\binom{a(x)-4}{b}3^{-b}$ is negligible. For instance, $x=2.10^{10}$ implies $\log(\log_2(x))+1=4.53,$ and $\sum_{b=1}^{5}\frac{3}{2}\binom{a(x)-4}{b}3^{-b}=4793$ (less than $\#\{n \in g(1), b(n)\leq 5, \; n<x\}= 5510$), that is a proportion $5.10^{-7}$ of the total of odd numbers less than $2.10^{10}$. This proportion tends to $0$ when $x$ tends to $\infty$, see the proposition\ref{loiNormale}.
   
    The condition C2 is false and some work has to be done to prove that the approximation made assuming C2, is sufficiently precise to conclude.
    Let $O(b,a)$ be the number of elements of $g_b^*(1)$ with $\sum_{i=1,b}v_i=a.$ Note that $M(b,x)=O(b,a(x))$.  Two approximations of $O(b,a)$ are now available:
\begin{itemize}
\item $O_1(b,a)= C(a,b,m)$ with $a \in (b,mb)$,
\item $O_2(b,a)= \frac{\binom{a-5}{b-1}}{m}$ with $a \in (b,m+b-1)$,
\end{itemize}
Note that $a \in (b,m+b-1) \Rightarrow O_2(b,a) \leqslant O_1(b,a).$ 

The proposed approximation for $M(b,x)$ is thus
 $$M_2(b,x)=\frac{3}{2}\frac{\binom{a(x)-4}{b}}{3^b},$$
   and the candidate lower bound for $\frac{\theta(x)}{2}$ is 
   $$M_2(x)=1+\frac{3}{2}\sum_{b=1}^{\infty}\binom{a(x)-4}{b}3^{-b}.$$
   By convention we assume that the odd number $1$ is obtained with $b=0$. 
\subsubsection{A toy example with b=5}
Let $b=5$. The 688 747 536 odd integers of $g_5^*(1)$ have been generated, and the values of $v_1$ and $a$ have been recorded. The left figure of table \ref{dista(1:mb)} gives the plot of $a$. The values of $O_1(b,a)$ have also been plotted on the same figure and the fit is so good that two curves cannot be separated.

\begin{table}[h]
\begin{tabular}{cc}
\includegraphics[width=7cm]{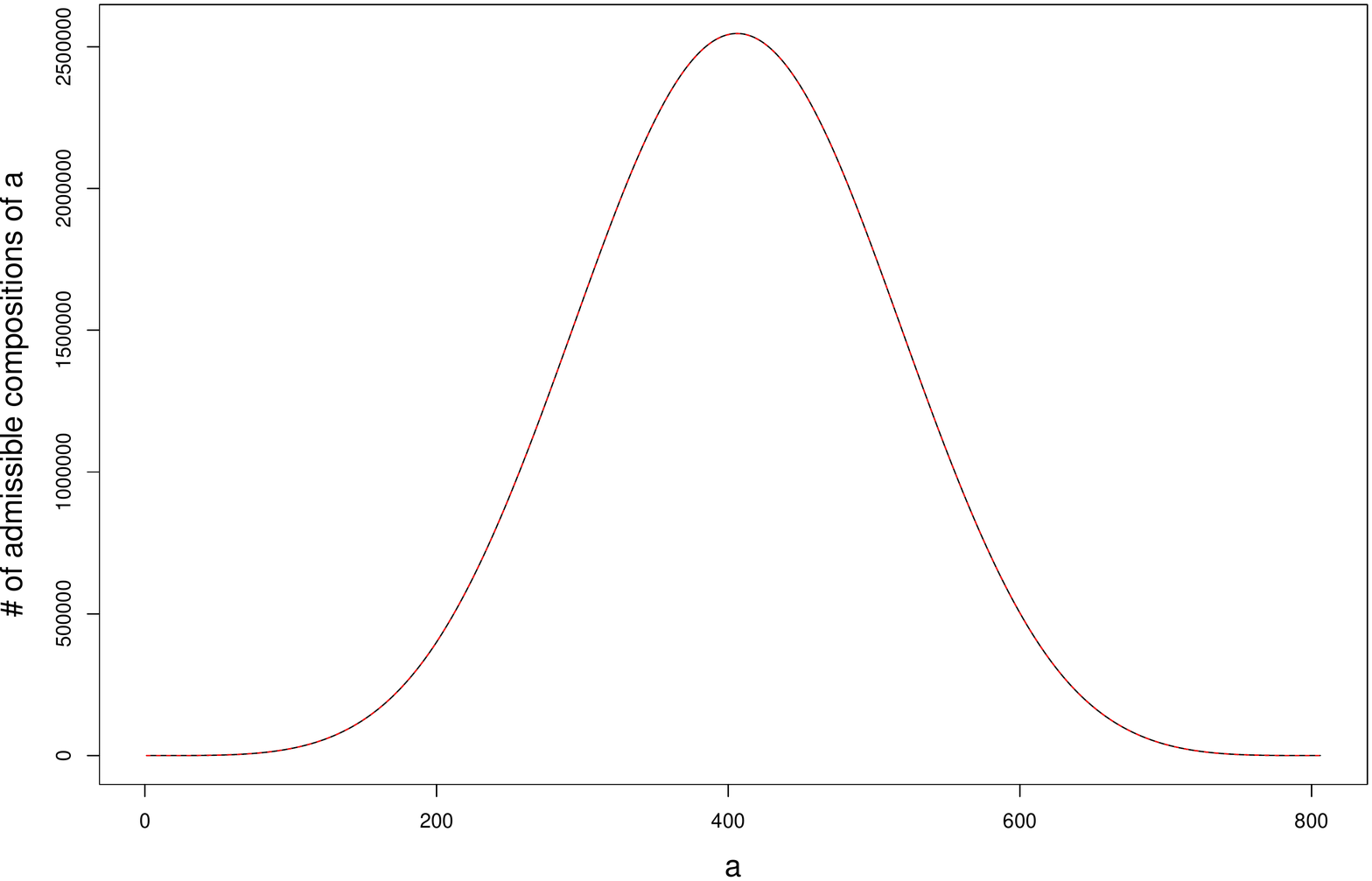} 
&
\includegraphics[width=7cm]{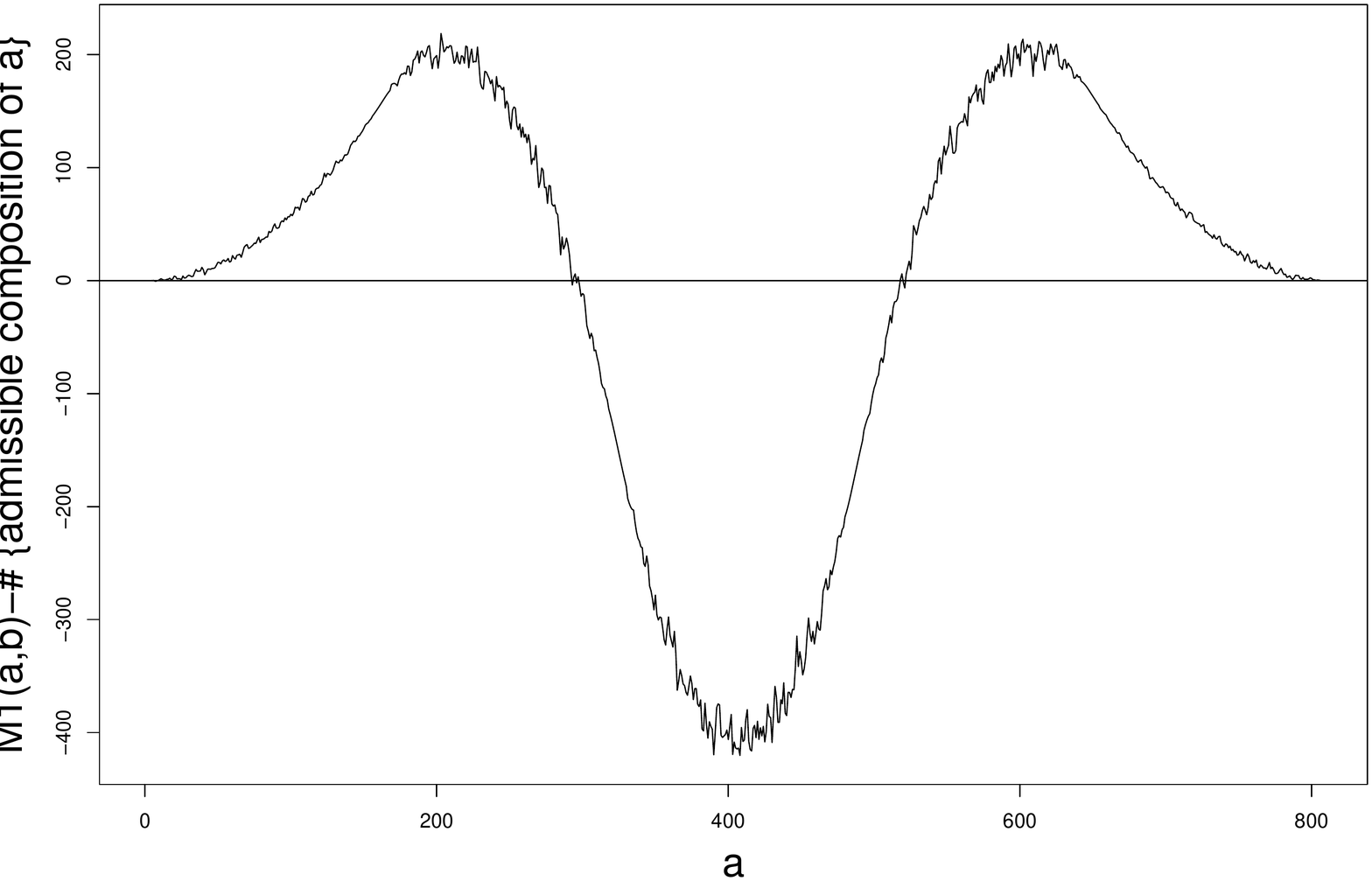} 
\end{tabular}
\caption{left : Number of odd integers for each $a$ obtained with $b=5$ steps, $O(5,a)$ (black continuous line), $O_1(5,a)$ (red dashed line), right: $O_1(5,a)$ - $O(5,a)$ }
\label{dista(1:mb)}
\end{table}

\begin{table}[h]
\begin{tabular}{cc}
\includegraphics[width=7cm]{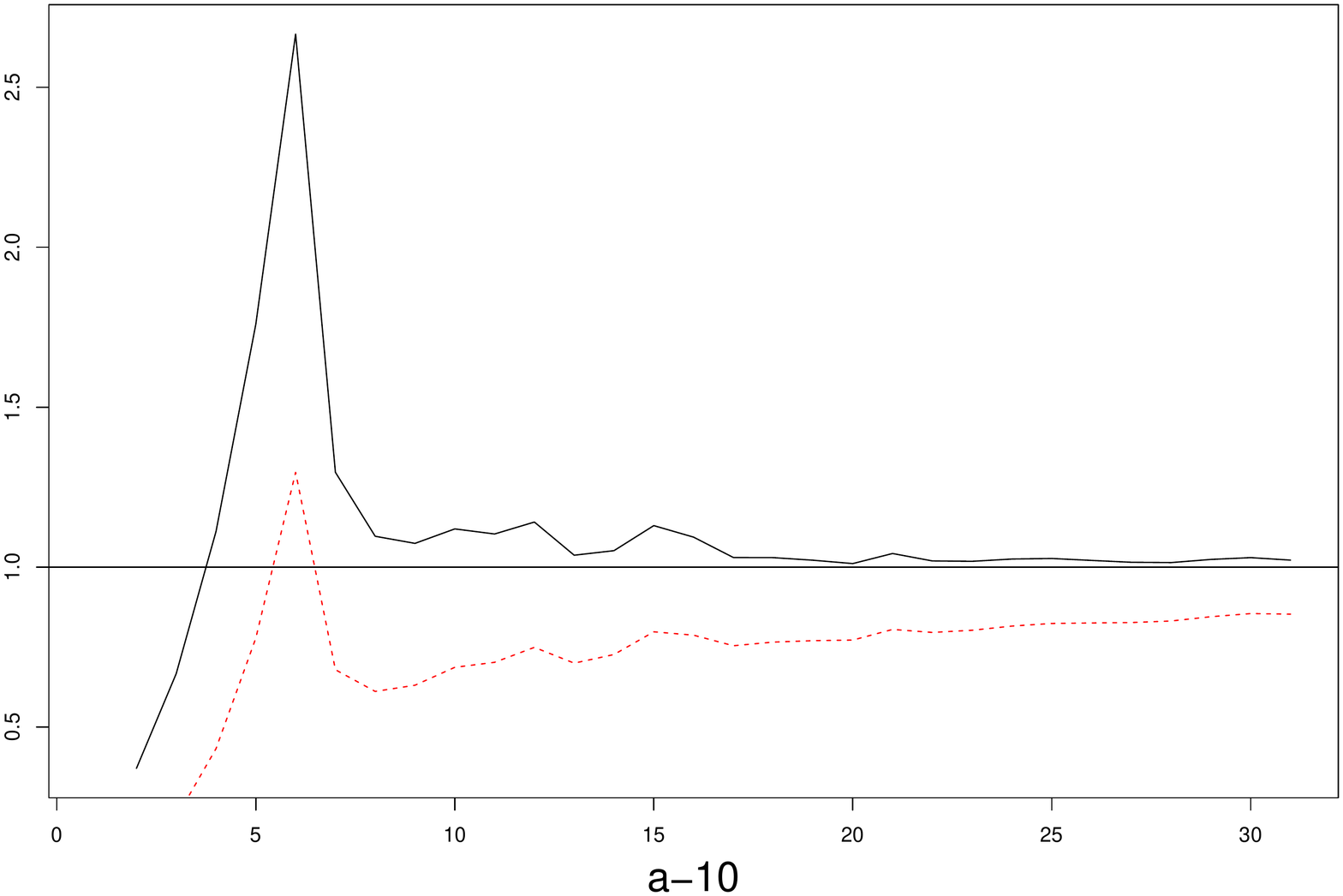} 
&
\includegraphics[width=7cm]{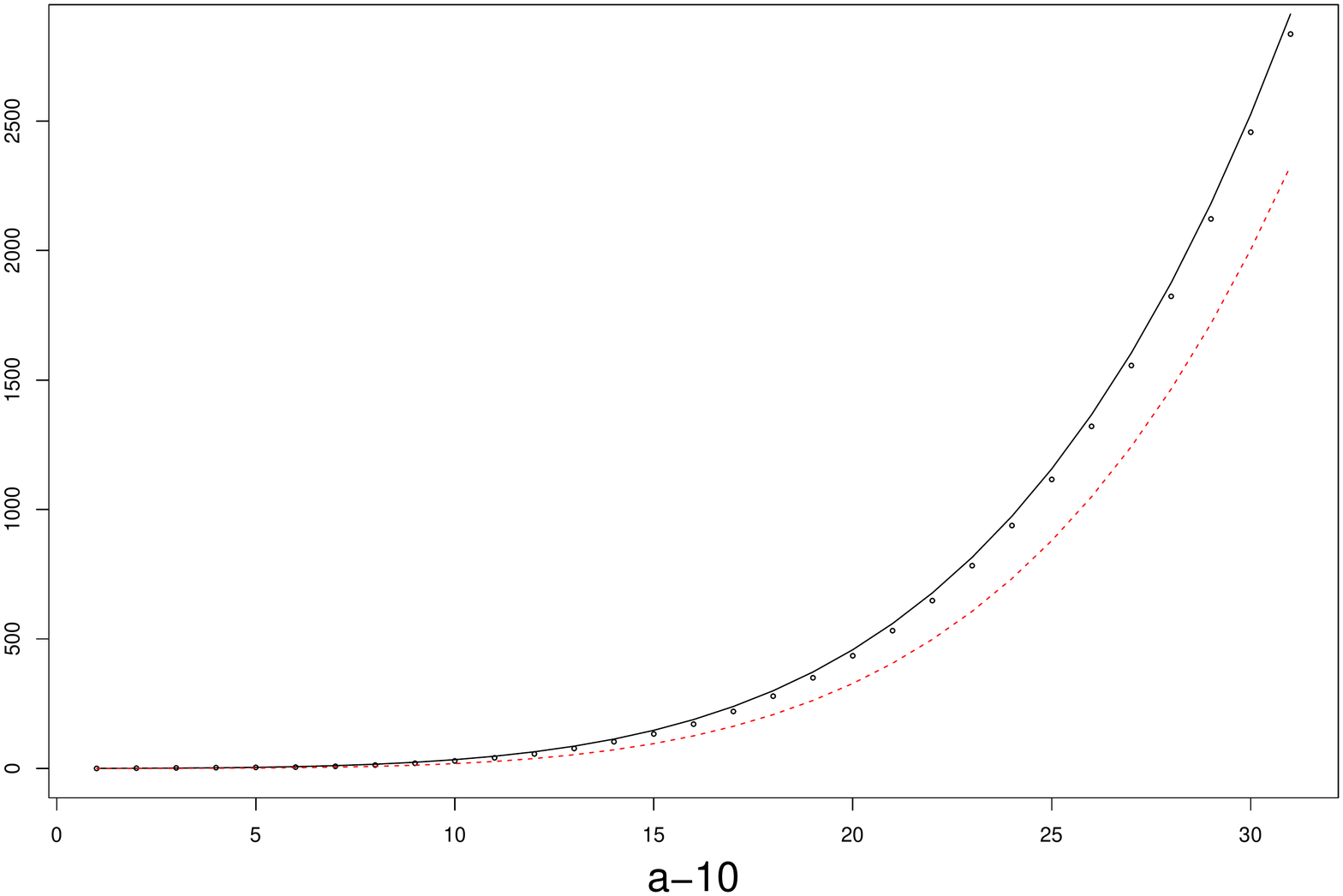} 
\end{tabular}
\caption{left : $O_1(5,a)/O(5,a)$ (black line) and $O_2(5,a)/O(5,a)$ (red dashed line) for $a \in (10,40)$, right: $\sum_{i=5,a}O(5,i)$ for $a \in (10,40)$ (circles), $\sum_{i=5,a}O_1(5,i)$  (black line) and $\sum_{i=5,a}O_2(5,i)$  (red dashed line) }
\label{cdfa}
\end{table}

However the right figure of table \ref{dista(1:mb)} shows that the two curves are not identical: the approximation overestimates the number of cases for extremal values of $a$ and underestimates the central values.
This result is expected because $a< b\log_2(3)$ is impossible for elements of $g_b^*(1)$ but possible for compositions of $a$. This implies that the lower tail of $a$ for elements of $g_b^*(1)$ is shorter than the lower tail of the sum of uniform distributions. The same is true for the upper tail by symmetry. The left figure of table \ref{cdfa} shows that, for $b=5$, the ratio $\frac{O_1(b,a)}{O(b,a)}$ is  largely greater than one for small $a$. $\frac{O_2(b,a)}{O(b,a)}< 1$ 
for most values of $a$ but not for all of them. The right figure of table \ref{cdfa} shows that $\sum_{i=5,a}O_1(b,i)$ overestimates $ \sum_{i=5,a}O(5,i)=\{ \# \{ n \in g(1), \; a(n) \leq a \}$ and $\sum_{i=5,a}O_2(b,i)$ is a better candidate for a lower bound.

 The figure \ref{a1v1} shows that condition C2 is false: $\alpha_b(4,a_1) < 1$ for low values of $a_1$ and $\alpha_b(4,a_1) > 1$ for high values of $a_1$. The pattern is opposite with $\alpha_b(10,a_1)$. These differences explain why the tail of $a$ is different from the tail of the sum of independent uniform variables: the smallest $v_1$ ($v_1=4$) is associated to higher values of $a_1$.

\begin{figure}[h] 
\begin{center}
\includegraphics[width=6cm]{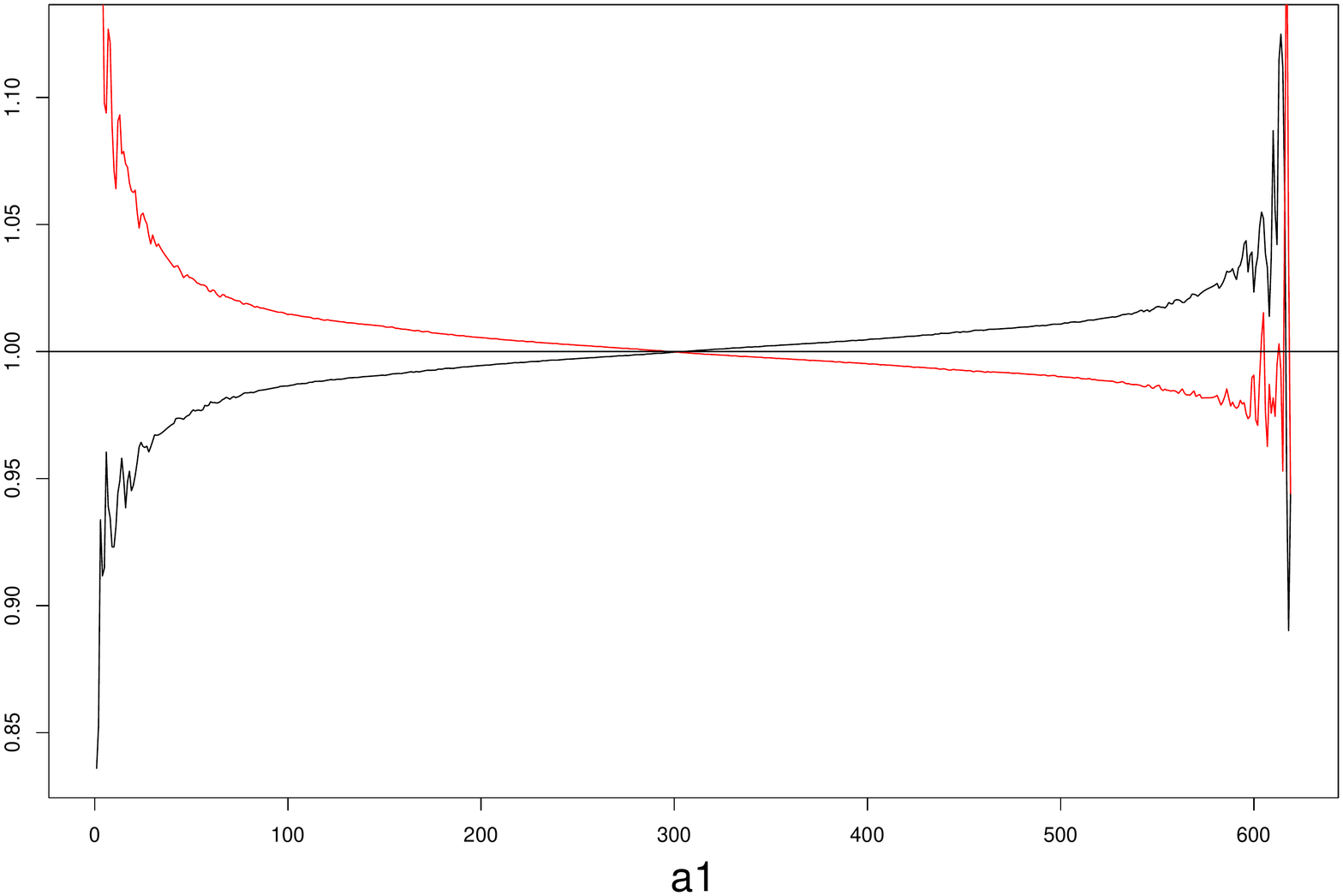} 
\caption{Plot of $\alpha_b(4,a_1)$ (black) and $\alpha_b(v_1=10,a_1)$ (red)}
\label{a1v1} 
\end{center}
\end{figure}

The variability of $N_b(v,a1)$ around its mean seems to be controlled. The figure \ref{EffectifV1} shows that the $54$ values of $\sum_{a_1=1,54}N_b(v,a1)$ with $v \in V,$ lie between 5174 and 6520. The values are clustered in 22 groups, 6 groups with one element and 16 with 3 elements. The values of $v_1$ for the 3-elements groups are separated by $54$: for example the group composed with $v_1=4,58,112$ is such that  $\sum_{a_1=1,54}N_b(v,a1)=5604.$ The mean of $\sum_{a_1=1,54}N_b(v,a1)$ is $5856.5=\binom{54}{4}3^{-4}$ and the standard deviation is equal to $433.2$. This pattern is produced by equation (\ref{admissible}). 

\begin{figure}[h] 
\begin{center}
\includegraphics[width=6cm]{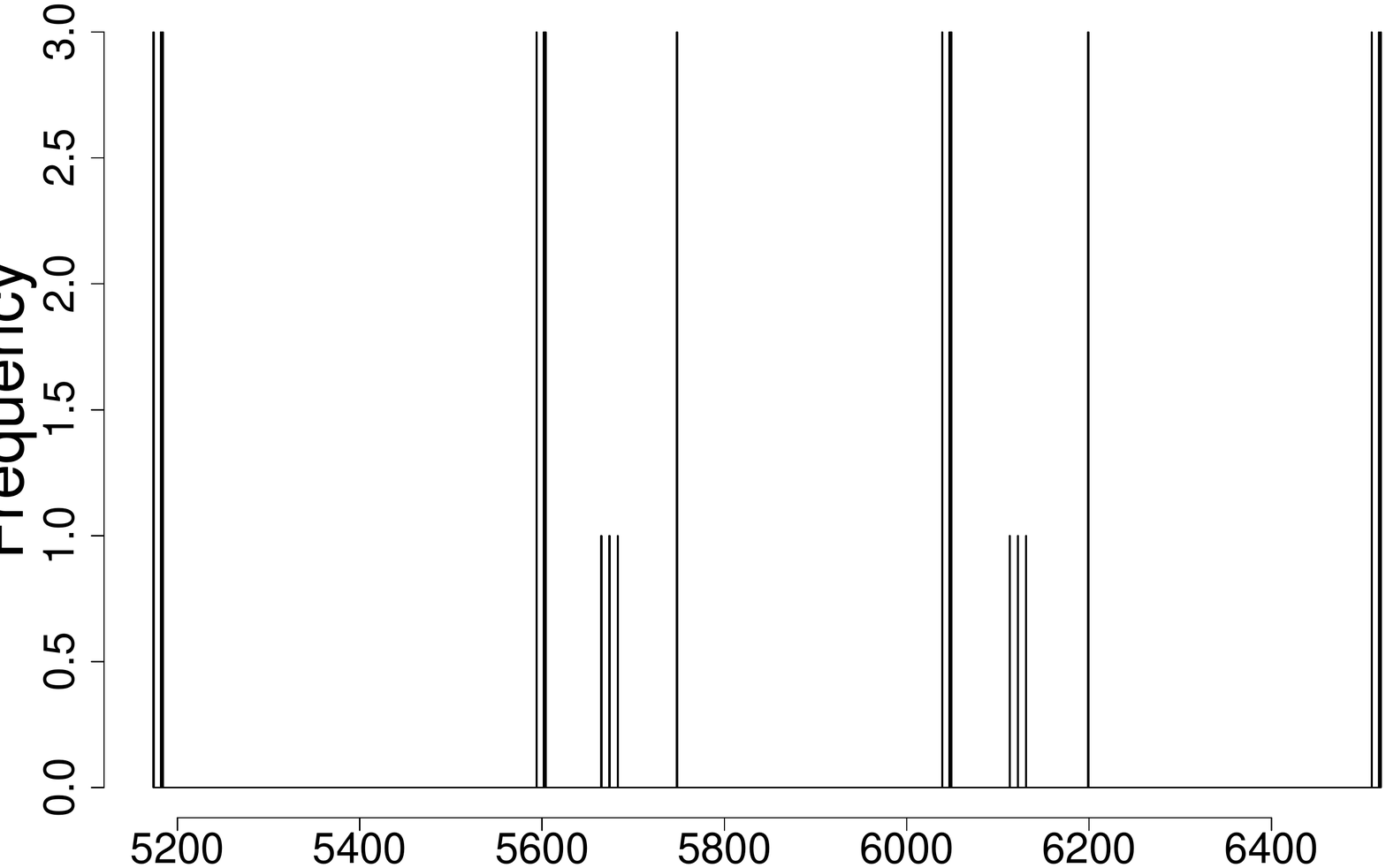} 
\caption{Histogram of $\sum_{a_1=1,54}N_b(v,a1)$  for $v \in V$}
\label{EffectifV1} 
\end{center}
\end{figure}

\subsection{Another lower bound with $v_1$=4}
 $v_1=4$ is the lower possible value of $v_1$ and gives $n=5$ for $b=1$. $v_2=1$ implies that $n=3$ for $b=2$. Therefore $v_2 \geq 3$ for all elements of $g_5^*(1)$ and $b>2$.
  \begin{eqnarray*}
  M(b,x) & \geq & \sum_{a_1 \leq a(x)-4-2}N(4,a_1) \\
         & \geq & \sum_{a_1 \leq a(x)-6}\alpha_b(4,a_1)\overline{N_b(a_1)} \\
         & \geq & 3/m\sum_{a_1 \leq a(x)-6}\alpha_b(4,a_1)\binom{b-1}{a_1}_m \\
  \end{eqnarray*}
 
   If Condition C1 ($a(x) < m+b-1$)  is true,
   \begin{eqnarray*}
   M(b,x) & \geq & \frac{3}{m}\sum_{a_1 \leq a(x)-6}\alpha_b(4,a_1)\binom{a_1-1}{b-2} \\
    & \geq & \frac{3}{m}\sum_{a_1=b-1}^{a(x)-6}\alpha_b(4,a_1)\binom{a_1-1}{b-2}\\
   \end{eqnarray*} 
    If $ \forall a_1, \; \alpha_b(4,a_1)=1$,
     
    \begin{eqnarray*}
   M(b,x) & \geq & \frac{3}{m}\sum_{a_1=b-1}^{a(x)-6}\binom{a_1-1}{b-2}\\
    & \geq & \frac{3}{m}\binom{a(x)-6}{b-1}\\
    & \geq & \frac{3}{2}\binom{a(x)-6}{b-1}3^{-b+1}
   \end{eqnarray*} 
  $M_3(x)=\sum_{b=1}^{\infty}\frac{3}{2}\binom{a(x)-6}{b-1}3^{-b+1}$ is thus another approximation of $M(x)$ and $M_3(x) < M_2(x)$.

\section{Step 3: closed form for $\sum_{b=1}^{\infty}\frac{3}{2}\binom{a(x)-4}{b}3^{-b}.$}
\begin{proposition}\label{Simplification}
$$\sum_{b=0}^{\infty} 3^{-b} \binom{b\log_2(3)+\log_2(x)}{b}=\frac{2x}{2-\log_2(3)} .$$
\end{proposition}

\begin{proof}
 The generalized binomial series $B_t(z)=\sum_{n=0}^{\infty}\binom{tn+1}{n}\frac{1}{tn+1}z^n$ with $n$ integer and $t,z,r$ real, has the following property, see ( \cite{Graham}, eq. 5.61):
 \begin{equation*}
 \frac{[B_t(z)]^r}{1-t+\frac{t}{B_t(z)}}=\sum_{n=0}^{\infty}\binom{tn+r}{n}z^n
 \end{equation*}
 
 Another property of $B_t(z)$ is given in   \cite{Graham}, eq. 5.59:
 \begin{equation*}\label{propriete B_t(z)}
 [B_t(z)]^{1-t}-[B_t(z)]^{-t}=z
 \end{equation*}
 that may be written $B_t(z)-1=z[B_t(z)]^{t}$\\
 Let $z=1/3$ and $t=\log_2(3)$, we obtain
 
  $B_{\log_2(3)}(1/3)-\frac{1}{3}[B_{\log_2(3)}(1/3)]^{{\log_2(3)}}=1.$  The  equation
  \begin{equation*}
  x-\frac{1}{3}x^{log_23}-1=0
  \end{equation*}
  possesses only two roots: $2$ and $4$.
  Therefore 
 \begin{eqnarray*}
 B_{\log_2(3)}(1/3) & = & \sum_{n=0}^{\infty}\binom{\log_2(3)n+1}{n}\frac{1}{\log_2(3)n+1}3^{-n}\\
 & = & 2.
 \end{eqnarray*}
 Therefore
 
 \begin{eqnarray*}
 \sum_{b=0}^{\infty} 3^{-b} \binom{b\log_2(3)+\log_2(x)}{b} & = & \frac{ \left [ B_{\log_2(3)}(1/3)\right ]^{\log_2(x)+1}}{B_{\log_2(3)}(1/3)(1-\log_2(3))+\log_2(3)}\\
  & = & \frac{2^{\log_2(x)+1}}{2(1-\log_2(3))+\log_2(3)}\\
  & = & \frac{1}{2-\log_2(3)} 2^{\log_2(x)+1}\\
 \end{eqnarray*}

\end{proof}

Now we have closed form expressions for $M_2(x)$ and $M_3(x)$:
\begin{proposition}
\begin{eqnarray}
M_2(x) & =& \frac{3}{16}\frac{1}{2-\log_2(3)}x\\
M_3(x) & =& \frac{9}{64}\frac{1}{2-\log_2(3)}x
\end{eqnarray}

\end{proposition}
\begin{proof}
 
\begin{eqnarray*}
M_2(x) & = & \frac{3}{2}\sum_{b=1}^{\infty} 3^{-b} \binom{b\log_2(3)+\log_2(x)-4}{b}+1 \\
& = & \frac{3}{2}\sum_{b=0}^{\infty} 3^{-b} \binom{b\log_2(3)+\log_2(x)-4}{b}-\frac{1}{2} \\
& = & \frac{3}{2}\frac{1}{2-\log_2(3)}2^{\log_2(x)-4+1} -\frac{1}{2} \\
& = & \frac{3}{16}\frac{1}{2-\log_2(3)}2^{\log_2(x)}-\frac{1}{2} \\
& = & 0.45177 x-\frac{1}{2}
\end{eqnarray*}

\begin{eqnarray*}
M_3(x) & = & \frac{3}{2}\sum_{b=1}^{\infty} 3^{-b+1} \binom{b\log_2(3)+\log_2(x)-6}{b-1} \\
& = & \frac{3}{2}\sum_{b=1}^{\infty} 3^{-b+1} \binom{(b-1)\log_2(3)+\log_2(3)+\log_2(x)-6}{b-1} \\
& = & \frac{3}{2}\sum_{b=0}^{\infty} 3^{-b} \binom{b\log_2(3)+\log_2(3)+\log_2(x)-6}{b} \\
& = & \frac{3}{2}\frac{1}{2-\log_2(3)}2^{\log_2(x)+\log_2(3)-6+1} \\
& = & \frac{9}{64}\frac{1}{2-\log_2(3)}2^{\log_2(x)} \\
& = & 0.3388 x
\end{eqnarray*}
\end{proof}
$M(x)$ is a lower bound for the number of odd integers included in $g(1)$. Therefore  $2M(x)$ is a lower bound for the number of integers included in $g(1)$. The expression $M_2(x)$ contains the term $-\frac{1}{2}$ that is negligible, and we forget it in the following.

\section{Toy example with $x=2.10^{10}$}
\begin{table}[h]
\begin{tabular}{cc}
\includegraphics[width=7cm]{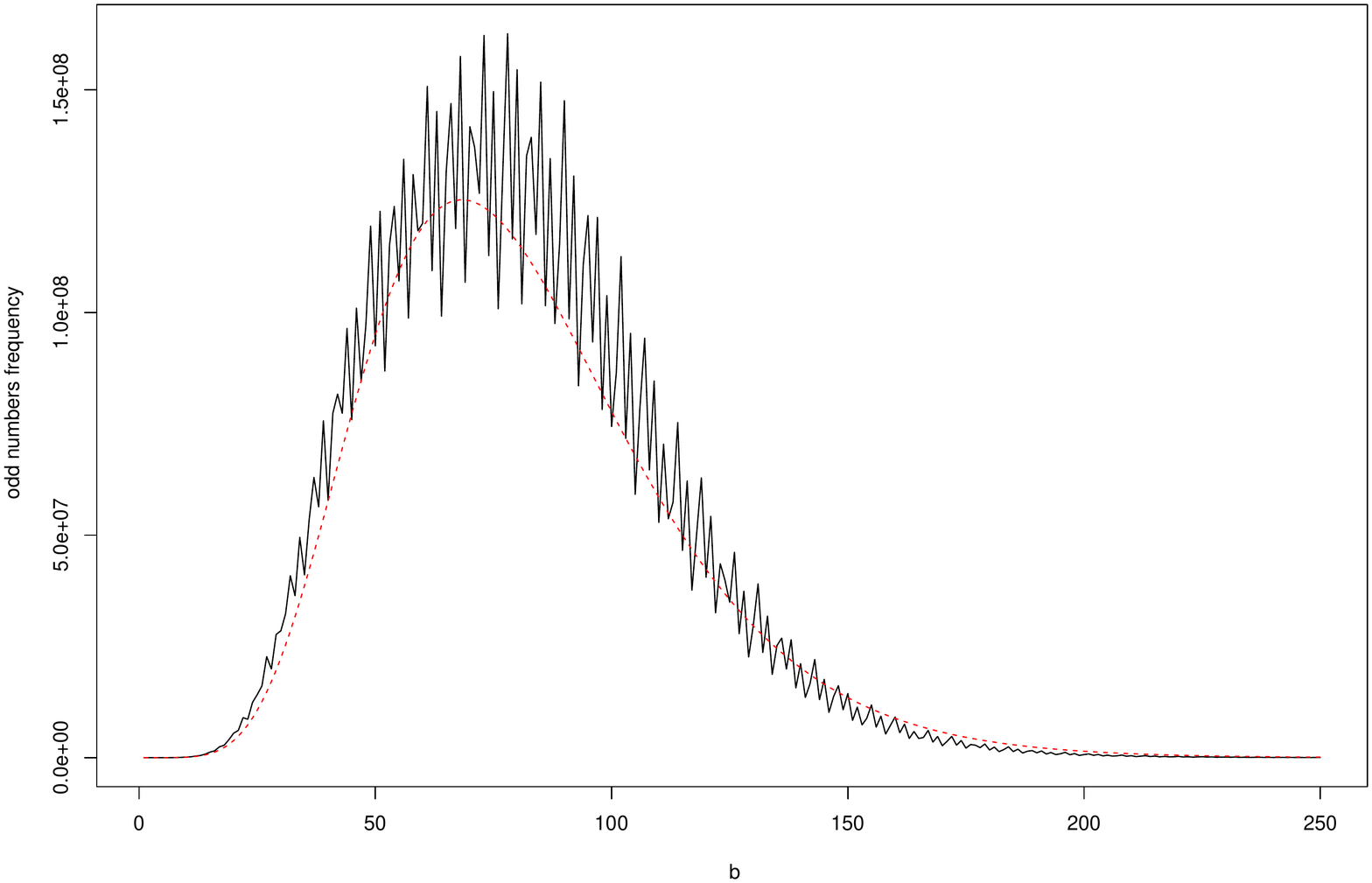} 
&
\includegraphics[width=7cm]{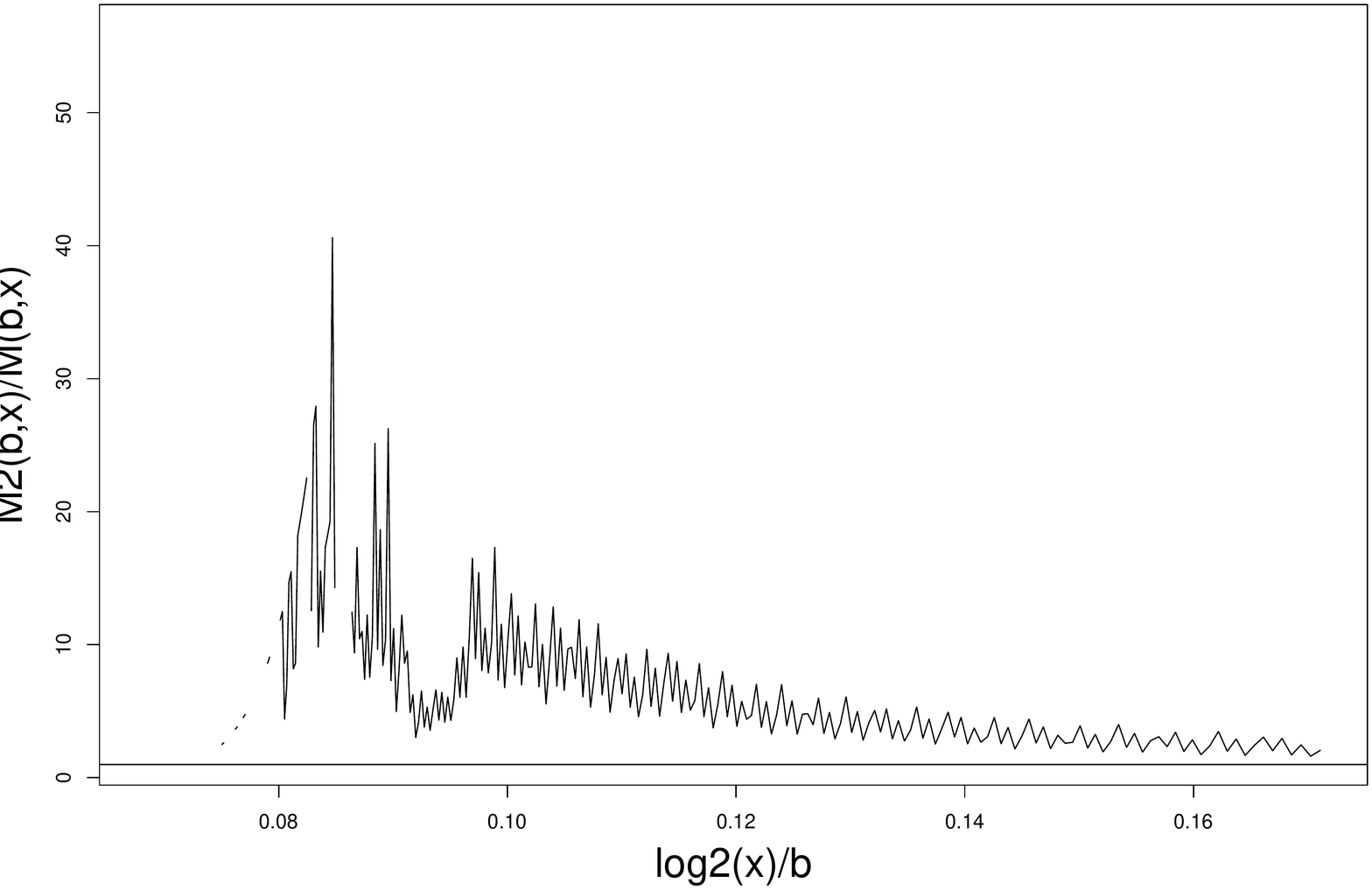} 
\end{tabular}
\caption{left : Number of odd integers less than $x=2.10^{10}$ obtained in $b$ steps ($b=1:300$) (black continuous line),  lowerbound $M_2$ (red dashed line) right: y-axis: $\frac{M_2(b,x)}{ M(b,x)}$, x-axis: $\frac{\log_2(x)}{b}$   ($b \in (200,500)$) }
\label{fig_x=2.10^{10}}
\end{table}

The values of $b$ for the $10^{10}$ odd numbers less than $x=2.10^{10}$ have been computed.  The left figure of table \ref{fig_x=2.10^{10}} shows the number of odd integers less than $x$ obtained in exactly $b$ steps, compared with $M_2(b,x)$. We have $M_2(b,x) < M(b,x)$ for $b<150$, but this is not true for $b>200$ (see the right figure of the same table), because the approximation $M_2(b,x)$ is not a lower bound for the extremal lower tail of the distribution of $a$.

However the overestimation of $M_2(b,x)$ for large $b$ is largely compensed by the underestimation of $M_2(b,x)$ for $b < 150$.  The figure \ref{fig_x=2.10^10Cum} shows that $\forall b \in \mathbb{N},\sum_{i=1}^b M_2(i,x) < \sum_{i=1}^b N(i,x)$.
 
 \begin{figure}[ht] 
 \begin{center}
\includegraphics[width=10cm]{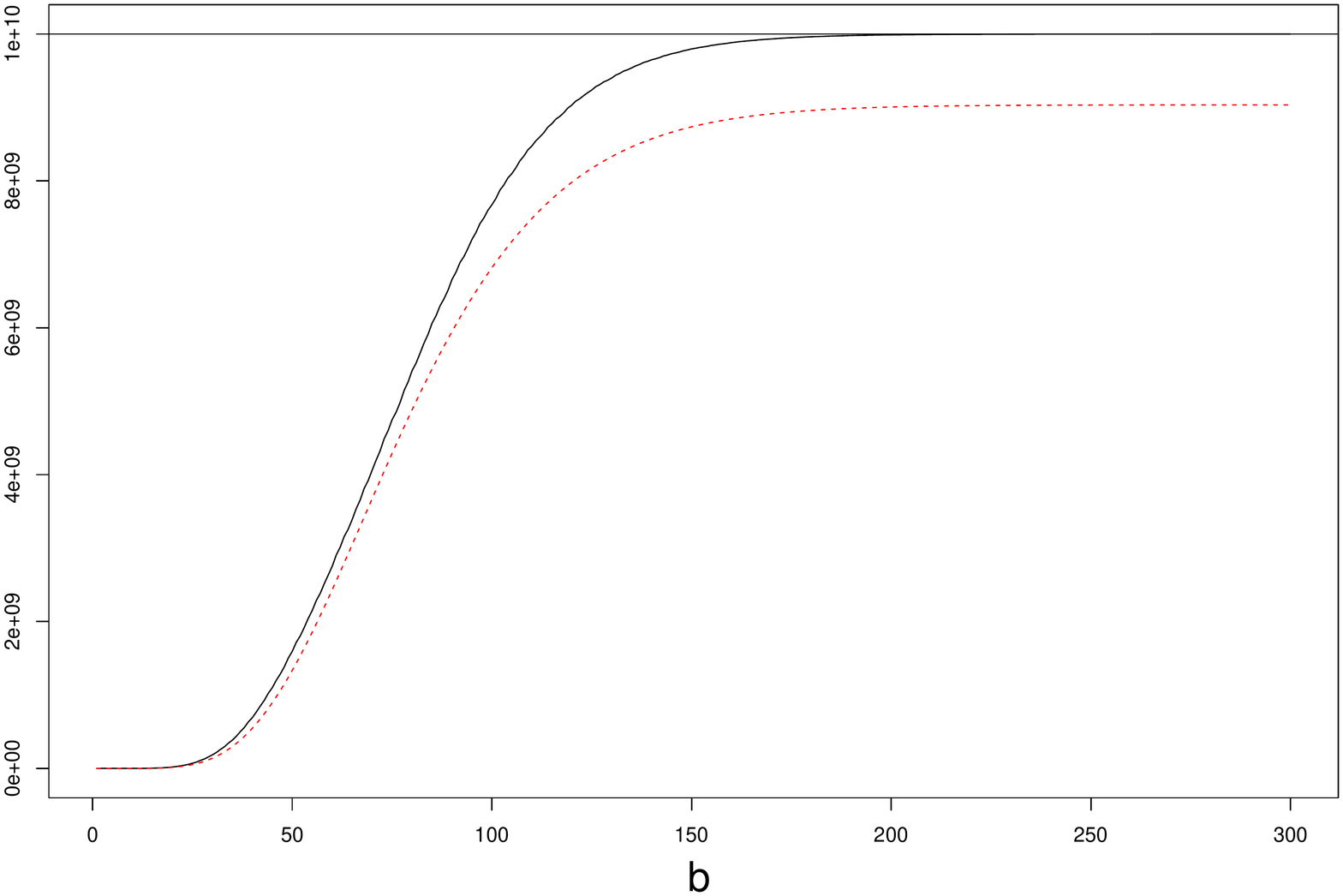} 
\caption{Number of odd integers less than $x=10^{10}$ obtained in less than $b$ steps (black continuous line)  lowerbound $M_2$ (red dashed line)}
\label{fig_x=2.10^10Cum} 
\end{center}
\end{figure}

\section{Properties of the distribution $P(B=b)= \frac{8(2-\log_2(3))}{x}\binom{a(x)-4}{b}3^{-b}$ for fixed $x$.}
For fixed $x$ let $B$ be an integer random variable defined by $P(B=b)= \frac{8(2-\log_2(3))}{x}\binom{a(x)-4}{b}3^{-b}$, the proportion $\frac{M_2(b,x)}{\sum_bM_2(b,x)}.$ There is no random process in the context of the Collatz problem. The probabilistic formalization is only a practical way to express the distribution of the values of $b$ for odd integers less than $x$. For instance, $\mathbb{E}(B)$ is an approximation of the mean value of $b$ among the odd integers less than $x$. The moments of $B$ can be expressed using the properties of the generalized binomial series $B_t(z)=\sum_{n=0}^{\infty}\binom{tn+1}{n}\frac{1}{tn+1}z^n$:
\begin{proposition}
$$\mathbb{E}_x(B)=\frac{log_2(x)}{2-log_2(3)}+\frac{5log_2(3)-8}{(2-log_2(3))^2}\simeq 2.409421 \log_2(x)-0.436487$$
\end{proposition} 

\begin{proof}
 $$z\frac{\partial{ \left\{ \frac{B_t(z)^r}{1-t+\frac{t}{B_t(z)}} \right\}}}{\partial z}=\sum_{n=0}^{\infty}\binom{tn+r}{n} n z^n,$$
 and
 $$
  z \frac{\partial{ \left\{ \frac{B_t(z)^{r+1}}{(1-t)B_t(z)+t}\right\} }}{\partial {z}}  = 
  z \left\{ \frac{(r+1)B_t(z)^{'}B_t(z)^r}{(1-t)B_t(z)+t} - \frac{(1-t)B_t(z)^{'}B_t(z)^{r+1}}{((1-t)B_t(z)+t)^2}\right\}
 $$
Moreover
$$ B_t(z)-1=zB_t(z)^t \Rightarrow B_t(z)^{'}=\frac{B_t(z)^t}{1-ztB_t(z)^{t-1}}.$$
$z=1/3$ and $t=\log_2(3) \Rightarrow B_{t}(z)^{'}=\frac{6}{2-\log_2(3)}.$

Therefore
$$\sum_{n=0}^{\infty}\binom{\log_2(3)n+r}{n} n \left( \frac{1}{3}\right )^n =\frac{2^{r+1}}{2-log_2(3)}\left( \frac{r+1}{2-log_2(3)}+\frac{2(log_2(3)-1)}{(2-log_2(3))^2} \right),$$
$r=\log_2(x)-4 \Rightarrow \mathbb{E}_x(B)=\frac{\log_2(x)}{2-log_2(3)}+\frac{5log_2(3)-8}{(2-log_2(3))^2}$

\end{proof}

\begin{proposition}\label{prop-variance}
$$\mathbb{V}_x(B)=\log_2(x)\frac{2}{(2-log_2(3))^3}+\frac{2\log_2(3)^2+8\log_2(3)-16}{(2-log_2(3))^4}\simeq 27.9749 \log_2(x)+57.4246$$
\end{proposition} 

\begin{proof}
see annex
\end{proof}
Higher moments can be computed by the same method. Moreover the pdf of $B$ tends to normality when $x$ tends to $\infty$.

\begin{proposition}\label{loiNormale}
$$\lim_{x\rightarrow\infty} \mathcal{L}\left( \frac{B-\mathbb{E}_x(B)}{\mathbb{V}_x(B)^{\frac{1}{2}}} \right)=\mathcal{N}(0,1)$$
\end{proposition} 

\begin{proof}
see annex
\end{proof}

The proposition \ref{loiNormale} may be used to cut the last terms of $M_2(x).$ Actually the proposed minoration $M_2(b,x)$ of $M(b,x)$ prove defective for high values of $b$ such that $\frac{\log_2(x)}{b} < 0.25 \Leftrightarrow b > 4\log_2(x)$. For large $x$:
$$M'_2(x)=\frac{3}{2}\sum_{b=\mathbb{E}_x(B)-2\sqrt{\mathbb{V}_x(B)}}^{\mathbb{E}_x(B)+2\sqrt{\mathbb{V}_x(B)}} 3^{-b} \binom{b\log_2(3)+\log_2(x)-4}{b} \simeq 0.954M_2(x)=\frac{3}{16}\frac{0.954}{2-\log_2(3)}x.$$

Moreover for large $x$,
\begin{eqnarray*}
\mathbb{E}_x\left(\frac{B}{\log_2(x)}\right)+2\sqrt{\mathbb{V}_x\left(\frac{B}{\log_2(x)}\right)} & = & 2.41- \frac{0.44}{\log_2(x)}+2\sqrt{\frac{27.9749}{ \log_2(x)}+\frac{57.4246}{(\log_2(x))^2}} \\
& < & 4 
\end{eqnarray*}
The concentration of the pdf of $\frac{B}{\log_2(x)}$ around $\mathbb{E}_x\left(\frac{B}{\log_2(x)}\right)$ implies that the problem of the defective large values of $b$ in $M_2(x)$ vanishes when $x \rightarrow \infty$. The same argument applies to the low values of $b<\log(\log_2(x))+1$ for which the condition C1 is not valid.

\bibliographystyle{plain} 
\bibliography{biblio}

\appendix
\section{Proof of proposition \ref{prop-variance}.}
 $$z^2\frac{\partial^2{ \left\{ \frac{B_t(z)^r}{1-t+\frac{t}{B_t(z)}} \right\}}}{\partial^2 z}=\sum_{n=0}^{\infty}\binom{tn+r}{n} n(n-1) z^n,$$
 and
 $$
   \frac{\partial^2{  \frac{B_t(z)^{r+1}}{(1-t)B_t(z)+t} }}{\partial^2 {z}}  = 
   \frac{\partial{\left\{ \frac{(r+1)B_t(z)^{'}B_t(z)^r}{(1-t)B_t(z)+t} - \frac{(1-t)B_t(z)^{'}B_t(z)^{r+1}}{((1-t)B_t(z)+t)^2}\right\}}}{\partial z}
 $$
 \begin{eqnarray*}
 \frac{\partial{ \frac{(r+1)B_t(z)^{'}B_t(z)^r}{(1-t)B_t(z)+t}}}{\partial z} & = & (r+1)\left[ \frac{rB_t(z)^{'2}B_t(z)^{r-1}+B_t(z)^{''}B_t(z)^r}{(1-t)B_t(z)+t}-\frac{B_t(z)^{'2}B_t(z)^r(1-t)}{((1-t)B_t(z)+t)^2}\right]
 \end{eqnarray*}
 \begin{eqnarray*}
 \frac{\partial{ \frac{(1-t)B_t(z)^{'}B_t(z)^{r+1}}{((1-t)B_t(z)+t)^2}}}{\partial z} & = & (1-t)\left[ \frac{B_t(z)^{''}B_t(z)^{r+1}+(r+1)B_t(z)^{'2}B_t(z)^{r}}{((1-t)B_t(z)+t)^2}\right] \\
 & - & 2(1-t)\frac{(1-t)B_t(z)^{'2}B_t(z)^{r+1}}{((1-t)B_t(z)+t)^3}
 \end{eqnarray*}
 
Moreover, with $c=(2-t),$
\begin{eqnarray*}
B_t(z)''& = & \left( \frac{B_t(z)^t}{1-ztB_t(z)^{t-1}} \right) '\\
 & = & \frac{tB_t(z)^{t-1}B_t(z)'}{1-ztB_t(z)^{t-1}}+B_t(z)^t\frac{tB_t(z)^{t-1}+zt(t-1)B_t(z)^{t-2}B_t(z)'}{((1-ztB_t(z)^{t-1})^2}\\
  & = & \frac{t2^{t-1}6/c}{1-(1/3)t2^{t-1}}+2^t\frac{t2^{t-1}+(1/3)t(t-1)2^{t-2}6/c}{((1-(1/3)t2^{t-1})^2}\\
   & = & \frac{t\frac{3}{2}6/c}{1-(1/3)t\frac{3}{2}}+3\frac{t\frac{3}{2}+(1/3)t(t-1)\frac{3}{4}6/c}{((1-(1/3)t\frac{3}{2})^2}\\
    & = & \frac{18t}{c^2}+3\frac{6tc+6t(t-1)}{c^3} \\
     & = & \frac{18t}{c^2}\left(1+\frac{c+(t-1)}{c}\right) \\
      & = & \frac{18t}{c^3}\left(2c+(t-1)\right) \\
       & = & \frac{18t}{c^3}(3-t) \\
        & = & \frac{18\log_2(3)(3-\log_2(3))}{(2-\log_2(3))^3} \\
\end{eqnarray*}

With $t=\log_23,z=\frac{1}{3} \Rightarrow  B_t(z)=2, B_t(z)'=\frac{6}{c}$ and $B_t(z)''=d$ we obtain
\begin{eqnarray*}
 \frac{\partial{ \frac{(r+1)B_t(z)^{'}B_t(z)^r}{(1-t)B_t(z)+t}}}{\partial z} & = &  (r+1) \left[\frac{r(\frac{6}{c})^22^{r-1}+d2^r}{c}-\frac{(\frac{6}{c})^22^r(1-t)}{c^2}\right] \\
  & = & \frac{(r+1)2^r}{c^4} \left( 18rc+dc^3-36(1-t)c \right)\\
  & = & \frac{(r+1)2^r}{c^4} \left( 18rc+18t(3-t)-36(1-t) \right) \\
   & = & \frac{18(r+1)2^r}{c^4} \left( r(2-t)+t(3-t)-2(1-t) \right) \\
    & = & \frac{18(r+1)2^r}{c^4} \left( rc+5t-t^2-2 \right) \\
     & = & 9\frac{ 2^{r+1}}{c^5} \left((r+1) (rc^2+(5t-t^2-2)c) \right) 
 \end{eqnarray*}
 and
 \begin{eqnarray*}
 \frac{\partial{ \frac{(1-t)B_t(z)^{'}B_t(z)^{r+1}}{((1-t)B_t(z)+t)^2}} }{\partial z} & = & (1-t) \left[ \frac{d2^{r+1}+(r+1)(1-t)(\frac{6}{c})^22^r}{((1-t)2+t)^2} - 2\frac{(1-t)(\frac{6}{c})^22^{r+1}}{((1-t)2+t)^3}\right] \\
& = &  \frac{(1-t)2^{r+1}}{c^5}\left( dc^3+18(r+1)c - 72(1-t) \right) \\
& = &  \frac{(1-t)2^{r+1}}{c^5}\left((18t(3-t)+18(r+1)c - 72(1-t) \right) \\
& = &  18\frac{(1-t)2^{r+1}}{c^5}\left(t(3-t)+(r+1)c - 4(1-t) \right) \\
& = &  18\frac{(1-t)2^{r+1}}{c^5}\left((r+1)c -t^2+7t-4 \right) \\
& = &  9\frac{2^{r+1}}{c^5}\left(2(1-t)((r+1)c -t^2+7t-4) \right)
 \end{eqnarray*}
 and
 \begin{eqnarray*}
 z^2\frac{\partial^2{ \left\{ \frac{B_t(z)^r}{1-t+\frac{t}{B_t(z)}} \right\}}}{\partial^2 z} &=& \frac{2^{r+1}}{c^5}\left( (r+1) (rc^2+(5t-t^2-2)c)-2(1-t)((r+1)c -t^2+7t-4) \right) \\
 &=& \frac{2^{r+1}}{c^5}\left( r(r+1)c^2 + c(r+1)(5t-t^2-2-2(1-t))-2(1-t)(-t^2+7t-4) \right) \\
  &=& \frac{2^{r+1}}{c^5}\left( r(r+1)c^2+c(r+1)(-t^2+7t-4)+2(t-1)(-t^2+7t-4) \right)
 \end{eqnarray*}

Therefore
$$\sum_{n=0}^{\infty}\binom{\log_2(3)n+r}{n} n(n-1) \left( \frac{1}{3}\right )^n =\frac{2^{r+1}}{c^5}\left(  r(r+1)c^2+(r+1)c(-t^2+7t-4)+2(t-1)(-t^2+7t-4)\right),$$
$r=\log_2(x)-4 \Rightarrow \mathbb{E}(B(B-1))=\frac{1}{c^4}\left( (l-4)(l-3))c^2+(l-3)c(-t^2+7t-4) +2(t-1)(-t^2+7t-4) \right)$
and with $l=\log_2x$,
\begin{eqnarray*}
\mathbb{V}(B)&=&\mathbb{E}(B(B-1)+\mathbb{E}(B)-\mathbb{E}(B)^2\\
&=&\frac{1}{c^4}\left( (l-4)(l-3)c^2+(l-3)c(-t^2+7t-4) +2(t-1)(-t^2+7t-4) \right)\\
&+&\frac{l}{c}+\frac{5t-8}{c^2}-(\frac{l}{c}+\frac{5t-8}{c^2})^2\\
&=&\frac{l(4-2t)+2t^2+8t-16}{c^4} 
\end{eqnarray*}

\section{Proof of proposition \ref{loiNormale}.}
Let $g_{B_x}(s)$ the generating function of $B.$ With $t=\log_23,$ $l=\log_2x-4$ and $c=2-t$, 
\begin{eqnarray*}
g_{B_x}(s) & = & \sum_{b=0}^{\infty}s^bP_x(B=b)\\
& = & \sum_{b=0}^{\infty}s^b\frac{8(2-\log_2(3))}{x}\binom{a(x)-4}{b}3^{-b}\\
& = & \frac{8(2-t)}{x}\sum_{b=0}^{\infty}\left(\frac{s}{3}\right)^b\binom{bt+\log_2x-4}{b}\\
& = & \frac{2-t}{2}2^{4-\log2x}\sum_{b=0}^{\infty}\left(\frac{s}{3}\right)^b\binom{bt+\log_2x-4}{b}\\
& = & \frac{c}{2}2^{-l}\sum_{b=0}^{\infty}\left(\frac{s}{3}\right)^b\binom{bt+l}{b}\\
& = & c2^{-(l+1)}\frac{[B_t(\frac{z}{3})]^{l+1}}{(1-t)B_t(\frac{z}{3})+t}
\end{eqnarray*}
$x\rightarrow \infty \Rightarrow \frac{\mathbb{E}(B)}{\mathbb{V}(B)^{\frac{1}{2}}} = \sqrt{\frac{lc}{2}}+o(l^{\frac{1}{2}}),$ and $\frac{1}{\mathbb{V}(B)^{\frac{1}{2}}} = c\sqrt{\frac{c}{2l}}+o(l^{-\frac{1}{2}}).$ 

Therefore the generating function of $ \frac{B-\mathbb{E}(B)}{\mathbb{V}(B)^{\frac{1}{2}}}$ is $h_x(s)=s^{-\sqrt{\frac{lc}{2}}}g_{B_x}(s^{c\sqrt{\frac{c}{2l}}}).$
\begin{eqnarray*}
h_x(s) & = & s^{-\sqrt{\frac{lc}{2}}}\frac{c}{2}2^{-l}\frac{[B_t(\frac{s^{c\sqrt{\frac{c}{2l}}}}{3})]^{l+1}}{(1-t)B_t(\frac{s^{c\sqrt{\frac{c}{2l}}}}{3})+t}
\end{eqnarray*}
Let $\phi_x(u)=h_x(e^u)$ with $y=l^{-1},$ one obtains
\begin{eqnarray*}
\phi_x(u) & = & e^{-u\sqrt{\frac{lc}{2}}}c2^{-l-1}\frac{[B_t(\frac{e^{uc\sqrt{\frac{c}{2l}}}}{3})]^{l+1}}{(1-t)B_t(\frac{e^{uc\sqrt{\frac{c}{2l}}}}{3})+t}\\
\log(\phi_x(u)) & = & -u\sqrt{\frac{lc}{2}}+(l+1)\left[\log[B_t(\frac{e^{uc\sqrt{\frac{c}{2l}}}}{3})-\log2\right]-\left[\log((1-t)B_t(\frac{e^{uc\sqrt{\frac{c}{2l}}}}{3})+t)-\log(c)\right]\\
\log(\phi_x(u)) & = & -u\sqrt{\frac{c}{2y}}+(\frac{1}{y}+1)\left[\log[B_t(\frac{e^{uc\sqrt{\frac{cy}{2}}}}{3})-\log2\right]-\left[\log((1-t)B_t(\frac{e^{uc\sqrt{\frac{cy}{2}}}}{3})+t)-\log(c)\right]
\end{eqnarray*}
Let $A=u\sqrt{\frac{c^3}{2}},$ $e^{A\sqrt{y}}=1+A\sqrt{y}+\frac{1}{2}A^2y+o(y)$, 

and $B_t(z/3)=B_t(1/3)+(z-1)B'_t(1/3)+\frac{1}{2}(z-1)^2B''_t(1/3)+o(z^2)$ the Taylor expansion of $B_t(z)$ in the neighborhood of one. 
\begin{eqnarray*}
B_t\left( \frac{e^{A\sqrt{y}}}{3} \right) & = & B_t(\frac{1}{3})+A\sqrt{y}+\frac{1}{2}A^2yB'_t(1/3)+\frac{1}{2}\left(A\sqrt{y}+\frac{1}{2}A^2y\right)^2B''_t(1/3)+o(y)\\
& = & 2+2\frac{A}{c}\sqrt{y}+\frac{A^2}{c^3}y(4-t)+o(y)\\
\frac{B_t\left( \frac{e^{A\sqrt{y}}}{3} \right)}{B_t(\frac{1}{3})}& = & 1+\frac{A}{c}\sqrt{y}+\frac{A^2}{2c^3}y(4-t)+o(y)\\
\log \left[\frac{B_t\left( \frac{e^{A\sqrt{y}}}{3} \right)}{B_t(\frac{1}{3})}\right] & = & \frac{A}{c}\sqrt{y}+\frac{A^2}{2c^3}y(4-t)-\frac{1}{2}\left( \frac{A}{c}\sqrt{y} \right)^2+o(y)\\
& = & \frac{A}{c}\sqrt{y}+\frac{A^2}{c^3}y+o(y)\\
(1+\frac{1}{y})\log \left[\frac{B_t\left( \frac{e^{A\sqrt{y}}}{3} \right)}{B_t(\frac{1}{3})}\right] & = &
\frac{A}{c}y^{-\frac{1}{2}}+\frac{A^2}{c^3}+o(1)\\
 & = &
u\sqrt{\frac{c}{2}}y^{-\frac{1}{2}}+\frac{u^2}{2}+o(1)\\
\log(\phi_x(u)) & = & -u\sqrt{\frac{c}{2y}}+u\sqrt{\frac{c}{2}}y^{-\frac{1}{2}}+\frac{u^2}{2}+o(1)\\
& = & \frac{u^2}{2}+o(1)
\end{eqnarray*}
Note that $\log((1-t)B_t(\frac{e^{uc\sqrt{\frac{cy}{2}}}}{3})+t)-\log(c)=O(y^{\frac{1}{2}})$ and does not contribute to the limit of $\log(\phi_x(u)).$

Therefore  $\lim_{x\rightarrow \infty} \phi_x(u)=e^{\frac{u^2}{2}}$ , the gaussian moment generating function.
\end{document}